\documentclass[a4paper,10pt]{article} 

\usepackage{url} 
\usepackage{geometry} 
\usepackage{amsmath}
\usepackage{amssymb} 
\usepackage{amsthm} 
\usepackage{relsize} 
\usepackage{color}
\usepackage{graphicx}
\usepackage{tikz}
\usepackage{caption,subcaption}
\usetikzlibrary{calc,arrows}

\usepackage{bm} 

\DeclareMathAlphabet\mathbfcal{OMS}{cmsy}{b}{n} 

\renewcommand{\vec}[1]{\mathbf{#1}}
\renewcommand{\d}{\operatorname{d}\!}
\newcommand{\mat}[1]{\mathbf{#1}}
\newcommand{\ten}[1]{\mathbfcal{#1}}
\renewcommand{\top}{\mathsmaller T}
\newcommand{\pinv}{\dagger}
\newcommand{\true}[1]{\overline{#1}}

\newcommand{\eg}{\emph{e.g., }} 
\newcommand{\ie}{\emph{i.e., }} 
\newcommand{\etc}{\emph{etc.}} 
\newcommand{\isdef}{:=} 

\theoremstyle{definition}
\newtheorem{lemma}{Lemma}

\newtheorem{theorem}{Theorem}
\theoremstyle{remark}
\newtheorem{remark}{Remark}

\theoremstyle{definition}
\newtheorem{iexample}{Example}
\newenvironment{example}%
  {\color{green!50!black}\begin{iexample}}%
  {\end{iexample}}

\everymath{\displaystyle} 

\title{Decoupling Multivariate Polynomials Using First-Order Information}
\author{Philippe Dreesen \quad Mariya Ishteva \quad Johan Schoukens\\ 
Vrije Universiteit Brussel, Dept.~ELEC\\
Pleinlaan 2, B-1050 BRUSSELS\\
{\tt philippe.dreesen@vub.ac.be}}
\date{}

\begin{document}
\maketitle

\begin{abstract}
We present a method to decompose a set of multivariate real polynomials into linear combinations of univariate polynomials in linear forms of the input variables. 
The method proceeds by collecting the first-order information of the polynomials in a set of operating points, which is captured by the Jacobian matrix evaluated at the operating points. 
The polyadic canonical decomposition of the three-way tensor of Jacobian matrices directly returns the unknown linear relations, as well as the necessary information to reconstruct the univariate polynomials. 
The conditions under which this decoupling procedure works are discussed, and the method is illustrated on several numerical examples.
\end{abstract}

\section{Introduction}
\subsection{Problem Statement}
The problem addressed in this paper is how to decouple a given set of multivariate real polynomials. 
Such a so-called decoupled representation expresses how the polynomials can be written as a linear combination of parallel univariate polynomials of linear forms of the input variables. 
Formally the problem can be stated as follows: 
Consider a set of $n$ multivariate real polynomials $f_i(u_1,\ldots,u_m)$, with $i=1,\ldots,n$, of total degree\footnote{The total degree is defined as the maximal sum of the exponents of the variables in a term.}$d$ in $m$ variables. 
We wish to obtain a decomposition of the form
$$ f_i(u_1,\ldots,u_m) = \sum_{j=1}^r w_{ij} \cdot g_j\left(\sum_{k=1}^m v_{kj} u_k\right), \quad \mbox{for } i=1,\ldots,n,$$
where $g_j(x_j)$ are univariate polynomials of degree at most $d$. 
Generally, each $f_i(\vec{u})$ contains $\textstyle{m+d \choose m}$ coefficients, of which many correspond to `coupled' monomials consisting of several variables $u_i$, \eg $u_1 u_2$, $u_1^2 u_3$, $u_2 u_3^3$, \etc 

The decoupling task is visualized in Figure~\ref{fig:blockfigure}, and can be compactly represented using matrix-vector notation. 
\begin{figure*}[!htb]
\begin{center}
\tikz \node [scale=0.80] {\begin{tikzpicture}
\node (u1) at (0,3) {$u_1$};
\node at (0,2.15) {$\vdots$};
\node (um) at (0,1) {$u_m$};
\draw [thick,fill=black!20, rounded corners=5pt] (1,0.5) rectangle (5,3.5); \node (F) at (3,2) {$\vec{f}(\vec{u})$};
\draw [->, thick, label=] (5,3) -- (5.65,3) node[right] {$y_1$};
\node at (5.95,2.15) {$\vdots$};
\draw [->, thick] (5,1) -- (5.65,1) node[right] {$y_n$};
\draw [->, thick] (u1) -- (1,3);
\draw [->, thick] (um) -- (1,1);
\end{tikzpicture}
\quad
\raisebox{6\height}{\Large$\leftrightarrow$}
\quad
\begin{tikzpicture}
\node (u1) at (0,3) {$u_1$};
\node at (0,2.15) {$\vdots$};
\node (um) at (0,1) {$u_m$};
\draw [thick] (1,0.5) rectangle (2,3.5); \node (L) at (1.5,2) {$\mat{V}^\top$};
\draw [->, thick] (u1) -- (1,3);
\draw [->, thick] (um) -- (1,1);
\node [shape=rectangle,draw,thick,fill=black!20,rounded corners=5pt] (g1) at (4,3) {$g_1(x_1)$};
\draw [->, thick, label=] (2,3) -- (g1) node[above,midway] {$x_1$};
\node at (4,2.15) {$\vdots$};
\node [shape=rectangle,draw,thick,fill=black!20,rounded corners=5pt] (gr) at (4,1) {$g_r(x_r)$};
\draw [->, thick] (2,1) -- (gr) node[above,midway] {$x_r$};
\draw [thick] (6,0.5) rectangle (7,3.5); \node (R) at (6.5,2) {$\mat{W}$};
\draw [->, thick] (g1) -- (6,3) node[above,midway] {$z_1$};
\draw [->, thick] (gr) -- (6,1) node[above,midway] {$z_r$};
\node (y1) at (8,3) {$y_1$};
\node at (8,2.15) {$\vdots$};
\node (yn) at (8,1) {$y_n$};
\draw [->, thick] (7,3) -- (y1);
\draw [->, thick] (7,1) -- (yn);
\end{tikzpicture}};
\end{center}
\caption{Decoupling problem. Find from the polynomial mapping $\vec{y}=\vec{f}(\vec{u})$ the transformations $\mat{V}$ and $\mat{W}$ and the mappings $g_i(x_i)$ that constitute the parallel structure $\vec{f}(\vec{u}) = \mat{W} \vec{g}(\mat{V}^\top \vec{u})$.}
\label{fig:blockfigure}
\end{figure*}
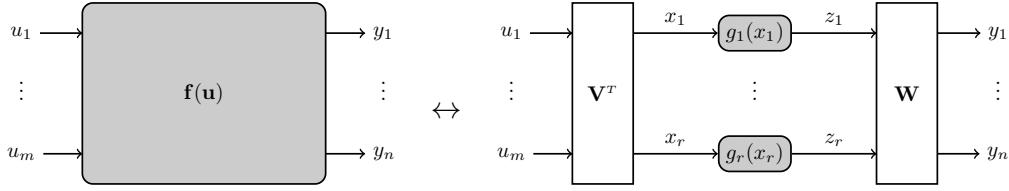
Consider therefore the multivariate polynomial vector function $\vec{f} : \mathbb{R}^m \to \mathbb{R}^n$ that is defined as 
$\vec{f}(\vec{u}) \isdef \left[ \begin{array}{ccc} f_1(\vec{u}) & \ldots & f_n(\vec{u}) \end{array} \right]^\top$,
in the variables $\vec{u} \isdef (u_1,\ldots,u_m)$. 
A \emph{decoupled representation} of $\vec{f}(\vec{u})$ is defined as 
\begin{equation}
    \vec{f}(\vec{u}) = \mat{W} \vec{g}(\mat{V}^\top \vec{u}),
    \label{eq:decoupled}
\end{equation}
where $\mat{V} \in \mathbb{R}^{m \times r}$ and $\mat{W} \in \mathbb{R}^{n \times r}$ are linear transformation matrices that relate the input variables $\vec{u}$ and the output variables $\vec{y}$ to the internal variables $\vec{x}, \vec{z} \in \mathbb{R}^r$ by the relations $\vec{x}=\mat{V}^\top \vec{u}$ and $\vec{y} = \mat{W} \vec{z}$, respectively. 
The function $\vec{g} \colon \mathbb{R}^r \to \mathbb{R}^r$ is defined as 
$$
\vec{g}(x_1,\ldots,x_r) = \left[ \begin{array}{ccc} g_1(x_1) & \ldots & g_r(x_r) \end{array} \right]^\top,
$$
with $g_k : \mathbb{R} \to \mathbb{R}$. 
The number $r$ corresponds to the number of internal univariate functions $g_i(x_i)$, and, as it will turn out, is closely related to the concept of \emph{tensor rank}, as will be discussed in Section~\ref{sec:method}.

The question may be studied in an exact or an approximate setting. 
For the exact case, the goal is to obtain an identical representation of a given set of polynomials, whereas in the non-exact case, an approximate representation (up to some degree of accuracy) is desired. 
The relevance of the question at hand is twofold: firstly, both in the exact and the non-exact setting, a decoupled representation may reveal new insights into a problem or may reduce the number of variables; secondly, in an approximate context, a decoupling may be useful to simplify a complex system. 

In this paper, we will study the decoupling task in the exact sense: we seek a decoupled representation that \emph{identically} matches a given set of multivariate polynomials and we assume that it exists.  
In order to clearly convey the ideas, in Example~\ref{ex:runningex1} we show a simple instance of two polynomials that have a decoupled representation.
Throughout the remainder of the paper, this simple example will be revisited to illustrate the results of the presented decoupling procedure. 
\begin{example}\label{ex:runningex1}
Consider the polynomials $f_1(u_1,u_2)$ and $f_2(u_1,u_2)$ of total degree $d=3$, given as
\begin{equation}
\begin{array}{rcl} 
    y_1 & = &  f_1(u_1,u_2)\\ \\
        & = & 54 u_1^3 - 54 u_1^2 u_2 + 8 u_1^2 + 18 u_1 u_2^2 + 16 u_1 u_2 - 2 u_2^3 + 8 u_2^2 + 8 u_2 + 1,\\ \\
   y_2 &=& f_2(u_1,u_2) \\ \\
       &=& - 27 u_1^3 + 27 u_1^2 u_2 - 24 u_1^2 - 9 u_1 u_2^2 - 48 u_1 u_2 - 15 u_1 + u_2^3 - 24 u_2^2 - 19 u_2 - 3.
\end{array}\label{eq:example}
\end{equation}
The equations~(\ref{eq:example}) were obtained from the following decoupled structure:
$$
\left[ \begin{array}{c} y_1 \\ y_2 \end{array} \right] =
\left[ \begin{array}{rr} 1 & 2 \\ -3 & -1 \end{array}\right]
\left[ \begin{array}{c} 2 x_1^2 - 3 x_1 + 1 \\ 2 x_2^3 - x_2 \end{array}\right]
,
\quad \textrm{with} \quad
\left[ \begin{array}{c} {x}_1 \\ {x}_2 \end{array} \right] =
\left[ \begin{array}{rr} -2 & -2 \\ -3 & -1 \end{array}\right]
\left[ \begin{array}{c} u_1 \\ u_2 \end{array}\right],
$$
revealing the internal univariate polynomials and the linear transformations at the input and output of the structure. 

It will turn out that a decoupled representation $\vec{f}(\vec{u}) = \mat{W} \vec{g}(\mat{V}^\top \vec{u})$ is not unique. 
To clearly make the distinction, the underlying representation (when it exists) will henceforth be denoted by barred symbols, \ie $\true{\vec{f}}(\vec{u}) = \true{\mat{W}} \true{\vec{g}}(\true{\mat{V}}^\top \vec{u})$, whereas the result of the decoupling procedure will be denoted by non-barred symbols, \ie $\vec{f}(\vec{u}) = \mat{W} \vec{g}(\mat{V}^\top \vec{u})$. 
Hence, for the polynomials~(\ref{eq:example}), we have 
$$
\begin{array}{rcl}
    \true{\mat{V}} &=& \left[ \begin{array}{rr} -2 & -3 \\ -2 & -1 \end{array} \right], \\ \\ 
    \true{\mat{W}} &=& \left[ \begin{array}{rr} 1 & 2 \\ -3 & 1 \end{array} \right], \\ \\ 
    \true{\vec{g}}(\true{\vec{x}}) &=& \left[ \begin{array}{r} \true{g}_1(\true{x}_1) \\ \true{g}_2(\true{x}_2) \end{array} \right] = \left[ \begin{array}{r} 2 \true{x}_1^2 -3\true{x}_1 + 1 \\ 2 \true{x}_2^3 - \true{x}_2 \end{array} \right].  
\end{array}
$$ 
\end{example}

\subsection{Related Work and Applications}
The problem at hand is related to the Waring problem for polynomials \cite{alexanderhirschowitz1995,iarrobino-kanev1999,landsberg2012tensorsgeometry,ranestad2000varsumpow} which concerns the decomposition of a single homogeneous multivariate polynomial $f(u_1,\ldots,u_m)$ of degree $d$ as 
$$ f(u_1,\ldots,u_m) = \sum_{i=1}^r w_i (v_1 u_1 + \cdots + v_m u_m)^d,$$
in which $r$ denotes the so-called Waring rank. 
Research on obtaining upper bounds on $r$, as well as developing algorithms for computing this decomposition dates back to Sylvester, who solved the case $m=2$ in 1886 \cite{sylvester1886}.
The Waring decomposition for $m > 2$ and several extensions of the problem have attracted research activity ever since (see \cite{alexanderhirschowitz1995,  comonmourrain1996, iarrobino-kanev1999, oeding2013eigtensors} and references therein).  
Today still, the problem receives a lot of research attention, especially due to the bijective relation between the homogeneous Waring decomposition and the symmetric tensor decomposition \cite{brachat2010std,comon2008symtensors,comonmourrain1996,kolda2009tdaa,landsberg2012tensorsgeometry,oeding2013eigtensors,schoukens2012ceipwsualit,usevich2014mtns}, of which the latter ---and tensor methods in general--- have become an important research domain in the last decades \cite{kolda2009tdaa}. 

The problem we study is very reminiscent of the classical Waring problem, however we consider the \emph{non-homogeneous case} of \emph{several polynomials}. 
The non-homogeneous Waring problem is studied in \cite{bialynicki2008,schinzel2002}.
The simultaneous Waring problem for several homogeneous polynomials is studied in \cite{carlini2003waringseveralforms,  tiels2013fctdpripwm}.  
In this paper we will restrict our attention to the case in which the Waring rank is low, and we focus on the computation of the decomposition. 

The decoupling task is of interest in non-linear block-oriented system identification \cite{giri2010bnsi} and non-linear state-space identification \cite{paduart2010} where it is often desired to recover the internal structure of an identified static non-linear mapping \cite{schoukens2012ceipwsualit,tiels2013fctdpripwm,vanmulders2013autom}. 
More generally, the task has connections with applications of tensor algebra methods in signal processing, see recent surveys \cite{cichocki2013spm,comon2014tabi} and references therein.

\subsection{Notation}\label{sec:notation} 
Scalars are denoted by lower-case or uppercase letters. 
Vectors are denoted by lower-case bold-face letters, \eg $\vec{x} \in \mathbb{R}^r$. 
Elements of a vector are denoted by lower-case letters with an index as subscript, \eg $\vec{x} = \left[ \begin{array}{ccc} x_1 & \ldots & x_r \end{array} \right]^\top$.
The Euclidean norm of a vector $\vec{x}$ is denoted as $\left\| \vec{x} \right\|$.
When a vector is passed to a function as an argument, the notation $\vec{x} \isdef (x_1,\ldots,x_r)$ is often used, \eg $y_1 = f_1(u_1,u_2)$ (see also below). 
Matrices are denoted by upper-case bold-face letters, \eg $\mat{V} \in \mathbb{R}^{m \times r}$. 
The entry in the $i$-th row and $j$-th column of the matrix $\mat{V}$ is $v_{ij}$, and we may represent a matrix $\mat{V}$ as $\textstyle \mat{V} = \left[ v_{ij} \right]$. 
A matrix $\mat{V} \in \mathbb{R}^{m \times r}$ can be represented by its columns as $\mat{V} = \left[ \begin{array}{ccc} \vec{v}_1 & \ldots & \vec{v}_r \end{array} \right]$.
The transpose and pseudo-inverse of a matrix $\mat{W}$ are denoted by $\mat{W}^\top$ and $\mat{W}^\pinv$, respectively. 
A diagonal matrix with diagonal elements $a_1$, $a_2$, $a_3$ is denoted by $\operatorname{diag}(a_1, a_2, a_3)$ or $\operatorname{diag}(a_i)$.
The rank of a matrix $\mat{A}$ is denoted as $\operatorname{rank}(\mat{A})$.
The dimension of the (right) null space of a matrix $\mat{A}$ is denoted by $\operatorname{dim} \operatorname{null} \mat{A}$. 
Higher-order tensors are $N$-way arrays and are denoted by bold-face upper-case caligraphical letters, \eg $\ten{J} \in \mathbb{R}^{n \times m \times N}$. 
The outer product is denoted by $\circ$ and defined as follows: For $\ten{X} = \vec{u} \circ \vec{v} \circ \vec{w}$, the entry in position $(i,j,k)$ is $u_i v_j w_k$. 
The Frobenius norm of a tensor $\ten{X}$ is denoted as $\left\| \ten{X} \right\|_F$.

For functions we employ the same convention as above.  
Scalar functions are denoted by lower-case symbols, \eg $f \colon \mathbb{R}^n \to \mathbb{R}$. 
Vector functions are denoted by lower-case bold symbols, \eg $\vec{f} \colon \mathbb{R}^m \to \mathbb{R}^n$, with $\vec{f}(\vec{u}) \isdef \left[ \begin{array}{ccc} f_1(u_1,\ldots,u_m) & \ldots & f_n(u_1,\ldots,u_m) \end{array} \right]^\top$.
Matrix functions are denoted by upper-case bold-faced symbols, \eg the Jacobian of $\vec{f}$ is denoted by $\mat{J} \colon \mathbb{R}^m \to \mathbb{R}^{n \times m}$ and is defined as
$
\mat{J}(\vec{u}) \isdef \left[ \partial f_i / \partial u_j (\vec{u}) \right]$.
The derivative of a univariate function $g(x)$ is often denoted using the simplified representation
$g'(x) \isdef \d g(x) / \d x$.
The ceiling function of a real number $x$ is denoted by $\left\lceil x \right\rceil$ and defined as the smallest integer not less than $x$.

\subsection{Outline of the Paper}
The remainder of this paper is organized as follows.
Section~\ref{sec:method} contains the description of the proposed approach that leads to a simultaneous matrix diagonalization problem, which is solved by a tensor decomposition. 
The method is presented and its properties are discussed.
In Section~\ref{sec:openprobs} we point out open problems for future work. 
Section~\ref{sec:conclusions} is devoted to the conclusions.

\section{Method}\label{sec:method}
\subsection{A Simultaneous Matrix Diagonalization Problem}\label{sec:matdiag}
The rationale behind the proposed method is to capture the behavior of $\vec{f}(\vec{u})$ by means of its first-order information collected in a set of operating points. 
The first-order information of a non-linear function $\vec{f}$ is contained in the Jacobian matrix of $\vec{f}(\vec{u})$, denoted by $\mat{J}(\vec{u})$ and defined as 
\begin{equation}
\mat{J}(\vec{u}) \isdef \left[ \begin{array}{ccc} \frac{\partial f_1}{\partial u_1}(\vec{u}) & \ldots & \frac{\partial f_1}{\partial u_m}(\vec{u}) \\ \vdots & \ddots & \vdots \\ \frac{\partial f_n}{\partial u_1}(\vec{u}) & \ldots & \frac{\partial f_n}{\partial u_m}(\vec{u}) \end{array} \right].
\label{eq:jacobian}
\end{equation}
By evaluating the Jacobian matrix in the operating points $\vec{u}^{(k)}$, $k=1,\ldots,N$, we will find that the decoupling task is solved by a simultaneous diagonalization of the set of Jacobian matrices $\mat{J}(\vec{u}^{(k)})$, obtained in this way.\footnote{In the classical literature, the simultaneous (or joint) diagonalization refers to the simultaneous congruence transformation $\mat{A}_k = \mat{V} \mat{D}_k \mat{V}^\top$, where the $\mat{A}_k$ are square matrices (see \cite{chabriel2014jointmatrixdiag} for a recent survey paper). In this paper, simultaneous diagonalization concerns the non-symmetrical problem $\mat{A}_k = \mat{W} \mat{D}_k \mat{V}^\top$, where, in addition to having different linear transformations on the left and on the right, the matrices $\mat{A}_k$ are not necessarily square.}
After the transformations $\mat{V}$ and $\mat{W}$ are determined, also an estimation of the internal univariate $g_i(x_i)$ can be reconstructed. 

\begin{lemma}\label{lem:jac}
The first-order derivatives of the parameterization~(\ref{eq:decoupled}) are given by
\begin{equation}
\mat{J}(\vec{u}) = \mat{W} \operatorname{diag} \left(g_i'(\vec{v}_i^\top \vec{u})\right) \mat{V}^\top,
\label{eq:jacobianfactorization}
\end{equation}
where $g_i'(x_i) \isdef \d g_i(x_i) / \d x_i$. 
\end{lemma}
\begin{proof}
The parameterization~(\ref{eq:decoupled}) is written more conveniently as 
$$
\vec{f}(\vec{u}) = \mat{W} \left[ \begin{array}{ccc} g_1(\vec{v}_1^\top \vec{u}) & \ldots & g_r(\vec{v}_r^\top \vec{u}) \end{array} \right]^\top,
$$
from which the lemma immediately follows by applying the chain rule. 
\end{proof}

Lemma~\ref{lem:jac} implies that the first-order derivatives of the parameterization~(\ref{eq:decoupled}), evaluated at the points $\vec{u}^{(k)}$, lead to the simultaneous diagonalization of a set of matrices $$ \mat{J}(\vec{u}^{(k)}) = \mat{W} \operatorname{diag}(g_i'(\vec{v}_i^\top \vec{u}^{(k)})) \mat{V}^\top,$$ in which the matrix factors $\mat{W}$ and $\mat{V}$ do not depend on the choice of the operating point $\vec{u}^{(k)}$.
Simultaneous matrix diagonalization can be computed by tensor methods. 

Consider the Jacobian tensor $\ten{J}$ that is constructed by stacking the Jacobian evaluations $\mat{J}(\vec{u}^{(k)})$ behind each other, giving rise to a three-way array of dimensions $n \times m \times N$.
The canonical polyadic decomposition (CP decomposition) \cite{carroll1970,harshman1970,kolda2009tdaa} expresses the tensor $\ten{J}$ as a sum of rank-1 terms.
The three-way tensor $\ten{J}$ is thus written as
\begin{equation}
    \ten{J} = \sum_{i=1}^r \vec{w}_i \circ \vec{v}_i \circ \vec{h}_i,
\label{eq:Jcpd}
\end{equation}
where $\circ$ denotes the outer product and $r$ is a positive integer. 
We have that
$$ 
\begin{array}{rcl} 
    \mat{W} &=& \left[ \begin{array}{ccc} \vec{w}_1 & \ldots & \vec{w}_r \end{array} \right],\\ \\
    \mat{V} &=& \left[ \begin{array}{ccc} \vec{v}_1 & \ldots & \vec{v}_r \end{array} \right], \quad \mbox{and}\\ \\
    \mat{H} &=& \left[ \begin{array}{ccc} \vec{h}_1 & \ldots & \vec{h}_r \end{array} \right],
\end{array}
$$
with $\mat{H}$ containing the evaluations of the $g_i'(\vec{v}_i^\top \vec{u})$ in the $N$ operating points as 
\begin{equation}
h_{ki} = g_i'(\vec{v}_i^\top \vec{u}^{(k)}).
\label{eq:hki}
\end{equation} 
Figure~\ref{fig:jaccpd} gives an overview of the simultaneous matrix diagonalization question and the CP decomposition. 
\begin{figure}[!htb]
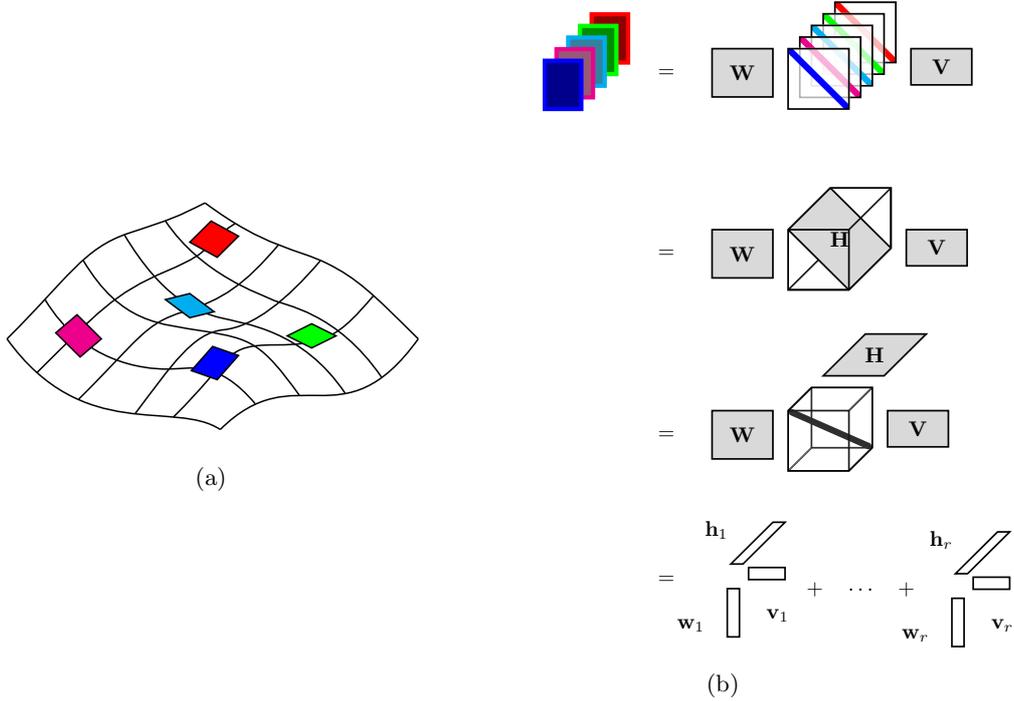

\begin{subfigure}[c]{0.45\textwidth}
\centering
\begin{tikzpicture}[shift={(0,5cm)},anchor=center] 
    \node[scale=0.38] {\input{fig-lineariz-arxiv.tex} };
\end{tikzpicture}
\caption{}
\end{subfigure}
\begin{subfigure}[c]{0.45\textwidth}
\centering
\begin{tikzpicture}[anchor=center,baseline] 
    \node[scale=0.8] {\input{fig-jaccpd-arxiv.tex}};
\end{tikzpicture}
\caption{}
\end{subfigure}
\caption{The first-order information of $\vec{f}(\vec{u})$ is collected in a set of operating points $\vec{u}^{(k)}$, with $k=1,\ldots,N$ (indicated by the colored patches on the surface shown in (a)). 
The corresponding Jacobian matrices $\mat{J}(\vec{u}^{(k)})$ are placed in a three-way tensor (b).
Lemma~\ref{lem:jac} states that each Jacobian matrix $\mat{J}(\vec{u}^{(k)})$ can be written as $\mat{J}(\vec{u}^{(k)}) = \mat{W} \operatorname{diag}(g_i'(\vec{v}_i^\top \vec{u}^{(k)})) \mat{V}^\top$. 
This results in a simultaneous matrix diagonalization problem, which is computed by the CP decomposition.}
\label{fig:jaccpd}
\end{figure}

\subsection{Uniqueness of the Canonical Polyadic Decomposition}\label{sec:identifiability}
Two aspects can easily be observed in the CP decomposition~(\ref{eq:Jcpd}) that prohibit the unique retrieval of the transformations $\mat{V}$ and $\mat{W}$ and the mappings $g_i(x_i)$. 
By rewriting (\ref{eq:Jcpd}) as $\textstyle \ten{J} = \sum_{i=1}^r (\alpha_i \vec{w}_i) \circ (\beta_i \vec{v}_i) \circ (\gamma_i \vec{h}_i)$, with $\alpha_i \beta_i \gamma_i = 1$, a column-wise scaling invariance becomes clear.
Additionally, the specific order in which the $r$ terms are collected into the factor matrices $\mat{V}$, $\mat{W}$ and $\mat{H}$ gives rise to an admissible permutation of the columns of the factors. 

The term \emph{essential uniqueness} is used to denote the uniqueness of the CP decomposition up to the column-wise scaling and permutation of the columns. 
Henceforth, we will use the term uniqueness when we refer to essential uniqueness. 
Kruskal \cite{kruskal1977,kruskal1988} has derived a condition that guarantees uniqueness of the CP decomposition. 
Essentially it provides an upper bound on the rank of a tensor in order to have a unique CP decomposition. 
We denote by $k_\mat{X}$ the \emph{Kruskal rank of a matrix $\mat{X}$}, which is defined as the largest number $k$ for which any set of $k$ columns of $\mat{X}$ is linearly independent.
\begin{theorem}[Kruskal \cite{kruskal1977,kruskal1988}]\label{thm:kruskal}
    The CP decomposition of $\ten{J}$ uniquely decomposes $\ten{J}$ into the factors $\mat{V}$, $\mat{W}$ and $\mat{H}$ (up to a permutation and scaling of the columns), provided that
\begin{equation}
k_\mat{V} + k_\mat{W} + k_\mat{H} \geq 2r + 2.
\label{eq:kruskal}
\end{equation}
\end{theorem}

It is often more practical to think of Theorem~\ref{thm:kruskal} in terms of the number of inputs $m$ and outputs $n$ of the non-linear function that we are decoupling.  
Under the assumption that the number of operating points is larger than the number of internal functions $g_i(x_i)$, \ie $N \geq r$, $\operatorname{rank} \mat{H} = r$, and $\mat{V}$ and $\mat{W}$ have full rank, which is often the case (\ie if the operating points $\vec{u}^{(k)}$ are chosen as random numbers), condition (\ref{eq:kruskal}) boils down to 
$$
\operatorname{min}(m,r) + \operatorname{min}(n,r) \geq r+2.
$$
It should be noted that condition~(\ref{eq:kruskal}) is quite reasonable in terms of number of inputs, outputs and number of internal $g_i(x_i)$.    

\begin{remark} 
Kruskal's uniqueness condition~(\ref{eq:kruskal}) does not imply that the optimization routine that computes the CP decomposition is not harmed by the problem of local minima: The result states that if the approximation error of the CP decomposition is zero, the retrieved factors $\mat{V}$, $\mat{W}$ and $\mat{H}$ are (up to a scaling and a possible permutation of the columns) identical to the underlying factors.  
\end{remark}

\subsection{Tensor Rank}\label{sec:tensorrank}
The integer $r$ has occurred in the above as the number of internal mappings $g_i(x_i)$ in the decoupled structure and as the number of terms in the CP decomposition. 
In the latter sense, the smallest integer $r$ for which (\ref{eq:Jcpd}) holds exactly, is the definition of the rank of the tensor $\ten{J}$. 
As opposed to the matrix rank, which is smaller than the smallest dimension, it is possible that $\operatorname{rank} \ten{J} > \operatorname{max}(m,n,N)$.
This also means that the number of internal mappings $g_i(x_i)$ may exceed the number of inputs and/or outputs. 

Determining the value for $r$ is a part of the decoupling procedure.
Currently, there are no direct ways to determine the (numerical) rank of a given tensor, although there exists notions of typical and generic rank of a tensor, as well as upper bounds on the rank, that are known for specific cases \cite{kolda2009tdaa}. 
It can be shown that
$ \operatorname{rank} \ten{J} \leq \operatorname{min}(mn, mN, nN)$ \cite{kolda2009tdaa}.
Note that, in practice, the number of operating points $N$ is typically chosen (much) larger than $m$ and $n$, in which case we have  $\operatorname{rank} \ten{J} \leq mn$.
In practice, the tensor rank $r$ is determined by assessing the approximation error of the rank-$r$ approximation of a tensor for consecutive values of $r$. 
For the exact decoupling task, assuming that the true transformations $\mat{V}$ and $\mat{W}$ meet Kruskal's uniqueness conditions (\ref{eq:kruskal}), the CP decomposition will indeed reach an approximation error that is sufficiently close to the machine precision when the correct $r$ is checked.

\begin{example}\label{ex:runningex2}
We revisit equations~(\ref{eq:example}) from Example~\ref{ex:runningex1}. 
We choose $N=2$ operating points $\vec{u}^{(k)}$ and their corresponding Jacobians $\true{\mat{J}}(\vec{u}^{(k)})$, as 
$$ \begin{array}{rclrcl} 
    \vec{u}^{(1)} &=& \left[ \begin{array}{r} -1\\0\end{array} \right], & \quad \true{\mat{J}}(\vec{u}^{(1)}) &=& \left[ \begin{array}{rr} 146 & -62 \\ -48 & 56 \end{array} \right], \\ \\
    \vec{u}^{(2)} &=& \left[ \begin{array}{r} 1\\-2\end{array} \right], & \quad \true{\mat{J}}(\vec{u}^{(2)}) &=& \left[ \begin{array}{rr} 434 & -158 \\  -192 & 104 \end{array} \right], 
\end{array}
$$
giving rise to a $2 \times 2 \times 2$ tensor $\true{\ten{J}}$. 
Since we know that $\true{r} = 2$ and Kruskal's uniqueness condition~(\ref{eq:kruskal}) guarantees uniqueness if $r \leq 2$, the choice $N=2$ is justified. 
It can be verified that the tensor can be decomposed using a rank-two CP decomposition up to a relative error of $1.64 \times 10^{-16}$, which confirms $r = 2$ as expected.  

The internal $\true{\vec{x}}^{(k)}$ can be computed using the expression $\true{\vec{x}} = \true{\mat{V}}^\top \vec{u}$, leading to 
$$ \true{\vec{x}}^{(1)} = \left[ \begin{array}{r} 2 \\ -3 \end{array}\right] \quad \mbox{and} \quad \true{\vec{x}}^{(2)} = \left[ \begin{array}{r} 2 \\ 5 \end{array} \right], $$ 
from which we can also compute the entries of $\true{\mat{H}}$ using $\true{h}_{ki} = \true{g}_i'(\true{x}_i^{(k)})$: 
$$\true{\mat{H}} = \left[ \begin{array}{rr} 5 & 26 \\ 5 & 74 \end{array} \right].$$
The CP decomposition is computed using tensorlab \cite{sorber2014tensorlab2} and returns three factors $\mat{V}$, $\mat{W}$ and $\mat{H}$, equal to the true factors up to a scaling and permutation of the columns as\footnote{Due to the lack of \emph{global} uniqueness, the numerical result of the CP decomposition may differ between executions, as well as when using a different routine for computing the CP decomposition.}
$$ 
\begin{array}{rclcl}
    \mat{V} &=& \left[ \begin{array}{rr} -512.1246 & -31.6350 \\  170.7082 & -31.6350 \end{array} \right] &=& \left[ \begin{array}{rr} -2 & 3 \\ -2 & -1 \end{array} \right] \left[ \begin{array}{rr} 0 & 15.8175 \\ -170.7082 & 0 \end{array}\right], \\ \\ 
    \mat{W} &=& \left[ \begin{array}{rr} -0.9189 &  -0.4470 \\ 0.4595  &  1.3411\end{array} \right] &=& \left[ \begin{array}{rr} 1 & 2 \\ -3 & 1 \end{array} \right] \left[ \begin{array}{rr} 0 & -0.4470 \\ -0.4595 & 0 \end{array}\right], \\ \\
    \mat{H} &=& \left[ \begin{array}{rr} 0.3315 & -0.7071 \\ 0.9435 & -0.7071 \end{array} \right] &=& \left[ \begin{array}{rr} 5 & 26 \\ 5 & 74 \end{array} \right] \left[ \begin{array}{rr} 0 & -0.1414 \\ 0.0127 & 0 \end{array} \right].  \end{array}
$$
It can easily be verified that the product of the scaling factors for the three factors yields unity for both columns.  
\end{example}

\subsection{Reconstructing the Internal Functions $g_i(x_i)$}
In this section we will describe how the coefficients of $g_i(x_i)$ are obtained from the retrieved $\mat{V}$ and $\mat{W}$, using input-output pairs $(\vec{u}^{(k)}, \vec{y}^{(k)})$.\footnote{Other methods exist to retrieve the coefficients of $g_i(x_i)$, for instance by using the fact that the factor $\mat{H}$ contains information about the differentiated $g_i(x_i)$ as in (\ref{eq:hki}). Such \emph{fitting and integration} methods may be of interest in the non-exact case, where the additional information can be helpful to obtain a better approximation, but they are not discussed in the current paper.} 
We can write each output $\vec{y}^{(k)}$ as a linear function of the coefficients $c_{i,j}$, and combine them into a block-equation system from which the coefficients can be determined. 
 
\subsubsection{Block-Vandermonde-like Linear System}
Recall that we have $\vec{y} = \mat{W} \vec{g}(\vec{x})$, which we write more conveniently as 
$$ 
\begin{array}{rcl}
             \vec{y}     &=& \mat{W} \left[ \begin{array}{c} c_{1,0} + c_{1,1} x_1 + \cdots + c_{1,d} x_1^d \\ \cdots \\ c_{r,0} + c_{r,1} x_r + \cdots + c_{r,d} x_r^d \end{array} \right],
\end{array}
$$
where $c_{i,j}$ denote the coefficients of the $i$-th polynomial $g_i(x_i)$.
The coefficients $c_{i,j}$ can then be combined into a single coefficient vector, leading to
\begin{equation}
\vec{y} = \mat{W} \left[ \begin{array}{@{}cccc|c|cccc@{}} 1 & x_1 & \ldots & x_1^d & & & & & \\ & & & & \ddots \\ & & & & & 1 & x_r & \ldots & x_r^d  \end{array} \right] \left[ \begin{array}{@{}c@{}} c_{1,0} \\ c_{1,1} \\ \vdots \\ c_{1,d} \\ \hline \vdots \\ \hline c_{r,0} \\ c_{r,1} \\ \vdots \\ c_{r,d} \end{array} \right], 
\label{eq:yWvdmcoef}
\end{equation}
where the empty entries correspond to zeros. 
Since $\mat{V}$ and $\mat{W}$ have been derived from the CP decomposition, we can compute $\vec{x}^{(k)}=\mat{V}^\top \vec{u}^{(k)}$ for a given operating point $\vec{u}^{(k)}$. 
By combining several instances of (\ref{eq:yWvdmcoef}) in this way for $k=1,\ldots,K$, the coefficients can be estimated using the linear system
\begin{equation}
\left[ \begin{array}{@{}c@{}} \vec{y}^{(1)} \\ \vdots \\ \vec{y}^{(K)} \end{array} \right] 
= 
\underbrace{\left[ \begin{array}{@{}ccc@{}} \mat{W} & & \\ & \ddots & \\ & & \mat{W} \end{array} \right] 
\underbrace{\left[ \begin{array}{cccc|c|cccc} 
1 & x_1^{(1)} & \ldots & (x_1^{(1)})^d & & & & & \\ 
& & & & \ddots & & & & \\ 
& & & & & 1 & x_r^{(1)} & \ldots & (x_r^{(1)})^d \\
\hline
\multicolumn{4}{c|}{\ldots} & & & & & \\ 
& & & & \ddots & & & & \\
& & & & & \multicolumn{4}{c}{\ldots}\\
\hline
1 & x_1^{(K)} & \ldots & (x_1^{(K)})^d & & & & & \\ 
& & & & \ddots & & & & \\ 
& & & & & 1 & x_r^{(K)} & \ldots & (x_r^{(K)})^d \\
\end{array} \right] }_{\mat{X}_K}
}_{\mat{R}_K}
\left[ \begin{array}{@{}c@{}} c_{1,0} \\ c_{1,1} \\ \vdots \\ c_{1,d} \\ \hline \vdots \\ \hline c_{r,0} \\ c_{r,1} \\ \vdots \\ c_{r,d} \end{array} \right],
\label{eq:yKRKc}
\end{equation}
where the empty entries represent (block) zeros.
The block-diagonal matrix with $\mat{W}$ blocks has size $K n \times K r$, the block-Vandermonde-like matrix $\mat{X}_K$ has size $K r \times r (d+1)$ and their product $\mat{R}_K$ has size $K n \times r (d+1)$. 
We introduce the short-hand notation $\vec{y}_K = \mat{R}_K \vec{c}$ as a compact way to represent (\ref{eq:yKRKc}).

\subsubsection{Existence and Uniqueness of Solutions}\label{sec:existuniq}
Let us investigate the existence and uniqueness aspects of (\ref{eq:yKRKc}), where we assume that $K$ is sufficiently large for the time being. 
Since the outputs $\vec{y}^{(k)}$ are constructed using $\vec{y} = \mat{W} \vec{g}(\mat{V}^\top \vec{u})$, it can be understood immediately that a solution of (\ref{eq:yKRKc}) always exists (in the exact sense). 

Understanding whether (\ref{eq:yKRKc}) has a \emph{unique} solution requires investigating the rank of $\mat{R}_K$.
The system $\vec{y}_K = \mat{R}_K \vec{c}$ has a unique solution if $\mat{R}_K$ has full rank. 
Let us have a closer look to see what happens when $\mat{R}_K$ is rank-deficient. 
Recall that $\mat{X}_K$ contains in its rows Vandermonde vectors constructed from the $x$-variables evaluated at $K$ operating points, which gives rise to the fact that the non-zero elements of certain columns consist of ones only. 
By reordering the columns of $\mat{X}_K$ (see (\ref{eq:yKRKc})) such that the columns containing the ones (corresponding to the constant terms $c_{i,0}$) are placed on the left, the system $\vec{y}_K = \mat{R}_K \vec{c}$ becomes $\vec{y}_K = \overline{\mat{R}}_K \overline{\vec{c}}$, where 
$$
\overline{\mat{R}}_K = \left[ \begin{array}{c|ccc} \mat{W} & \bm{\times} & \ldots & \bm{\times} \\ 
 \vdots & \vdots &  & \vdots \\ \mat{W} & \bm{\times} & \ldots & \bm{\times} \end{array} \right]$$ 
is the column-reordered version of $\mat{R}_K$ and $\overline{\vec{c}}$ represents the corresponding reordered coefficient vector. 
A consequence is that the matrices $\mat{W}$ and $\overline{\mat{R}}_K$ (and hence $\mat{R}_K$) have the same column rank-deficiency: we immediately see that the block-column containing the matrices $\mat{W}$ has the same column rank as $\mat{W}$; the right-hand-side part of $\overline{\mat{R}}_K$ (represented using the entries $\bm{\times}$) contains the powers of the $x_i^{(k)}$ and has full column rank, given that the operating points $\vec{u}^{(k)}$ are taken sufficiently persistent. 

Rank-deficiency occurs for example when there are fewer outputs $n$ than branches $r$, so that $r - \operatorname{rank} \mat{W}$ coefficients $c_{i,0}$ can be chosen freely, while $\vec{y}_K = \mat{R}_K \vec{c}$ remains exactly solvable. 
Notice that the `free parameters' are the constant terms $c_{i,0}$ only, as they correspond to the columns that form the $\mat{W}$ block in $\overline{\mat{R}}_K$.

The above considerations give rise to a straightforward way to determine the minimal number of operating points $K$ that is required to obtain an exactly solvable system.
The system (\ref{eq:yKRKc}) should become sufficiently overdetermined, meaning that the number of rows of $\mat{R}_K$ should be at least equal to the rank of $\mat{R}_K$. 
We have thus $K n \geq \operatorname{rank} \mat{R}_K$, which directly leads to the condition 
\begin{equation}
K \geq \left\lceil \frac{r(d+1) - \operatorname{dim} \operatorname{null} \mat{W} }{n} \right\rceil.
\label{eq:Kmin}
\end{equation}

\begin{example}
    \label{ex:runningex3}
We revisit once again equations~(\ref{eq:example}) and show how the $g_i(x_i)$ are reconstructed. 
We compute the minimal value for $K \geq 4$ using the formula~(\ref{eq:Kmin}) and choose $K=4$ linearization points $\vec{u}^{(k)}$, the corresponding outputs $\vec{y}^{(k)}$ and the internal variables $\vec{x}^{(k)} =  \mat{V}^\top \vec{u}^{(k)}$ as 
$$ \begin{array}{rclrclrcl} 
    \vec{u}^{(1)} &=& \left[ \begin{array}{r} -0.20 \\0\end{array} \right], & \quad \vec{y}^{(1)} &=& \left[ \begin{array}{r} 0.8880 \\ -0.7440 \end{array} \right], & \quad \vec{x}^{(1)} &=& \left[ \begin{array}{r} 102.4249 \\  6.3270  \end{array} \right], \\ \\
    \vec{u}^{(2)} &=& \left[ \begin{array}{r} 0.25 \\ -2.00 \end{array} \right], & \quad \vec{y}^{(2)} &=& \left[ \begin{array}{r} 51.0938 \\ -63.0469 \end{array} \right], & \quad \vec{x}^{(2)} &=& \left[ \begin{array}{r}  -469.4476 \\  55.3612  \end{array} \right], \\ \\
    \vec{u}^{(3)} &=& \left[ \begin{array}{r} 0.50 \\ 0.25 \end{array} \right], & \quad \vec{y}^{(3)} &=& \left[ \begin{array}{r} 11.4063 \\ -30.7032 \end{array} \right], & \quad \vec{x}^{(3)} &=& \left[ \begin{array}{r} -213.3853\\ -23.7262  \end{array} \right], \\ \\
    \vec{u}^{(4)} &=& \left[ \begin{array}{r} 0 \\0.50 \end{array} \right], & \quad \vec{y}^{(4)} &=& \left[ \begin{array}{r} 6.7500 \\ -18.3750 \end{array} \right], & \quad \vec{x}^{(4)} &=& \left[ \begin{array}{r}  85.3541 \\ -15.8175 \end{array} \right]. \\ \\
\end{array}
$$
We construct the $8 \times 8$ matrix $\mat{R}_K$, having rank $8$. 
Solving~(\ref{eq:yKRKc}) returns the coefficients $c_{i,j}$ and we find 
$$ 
 \begin{array}{rcl} g_1(x_1) &=&  -0.0127 x_1^3 -4.3751 \times 10^{-7} x_1, \\ \\ g_2(x_2) &=& -0.0179 x_2^2 +0.4243 x_2 -2.2369.\end{array}   
$$
We verify that $\true{\vec{f}}(\vec{u})$ corresponds to $\vec{f}(\vec{u})$ up to a relative error on the coefficients (\ie $\| \vec{c} - \true{\vec{c}} \| / \| \true{\vec{c}} \|$) of $0.0925 \times 10^{-14}$ for $f_1$ and $0.1302 \times 10^{-14}$ for $f_2$.  
We notice that $g_i(x_i) \neq \true{g}_i(\true{x}_i)$; in Section~\ref{sec:relationgis} we will discuss the exact relation between the coefficients of $g_i(x_i)$ and $\true{g}_i(\true{x}_i)$. 
\end{example}

\begin{example}\label{ex:fatWex4}
We present an example for which $m=n=3$ and $r=4$, in which the matrix $\mat{W}$ is column rank-deficient. 
Consider the equations
$$ 
\begin{array}{rcl}
    \true{f}_1(u_1, u_2, u_3) &=& - 4 u_1^2 + 8 u_1 u_3 + 6 u_1 - 3 u_3^2 - 8 u_3 - 6, \\ \\    
    \true{f}_2(u_2, u_2, u_3) &=&  2 u_1^2 - 4 u_1 u_3 - 3 u_1 + u_2^3 + 6 u_2^2 u_3 + 12 u_2 u_3^2 - u_2 + 8 u_3^3 + 2 u_3^2 + u_3 + 3, \\ \\ 
      \true{f}_3(u_1,u_2,u_3) &=& - 2 u_1^2 + 4 u_1 u_3 + 4 u_1 - 2 u_3^2 - 3 u_3 - u_2 - 8,
\end{array}
$$ 
which were obtained as $\true{f}(\vec{u}) = \true{\mat{W}} \true{\vec{g}}(\true{\mat{V}}^\top \vec{u})$ with 
$$ 
\begin{array}{rcl}
    \true{\mat{V}} &=& \left[ \begin{array}{rrrr}   1  &   0  &  0   & 1 \\ 0  &  1  &  0  & -1 \\  -1  &  2  &  1  &  0 \end{array} \right], \\ \\
    \true{\mat{W}} &=& \left[ \begin{array}{rrrr}  -2  &  0  &  1  &  0 \\  1  &  1  &  0  &  0 \\   -1  &  0  &  0  &  1 \end{array} \right], \quad \mbox{and} \\ \\
    \true{\vec{g}}(\true{\vec{x}}) &=& \left[ \begin{array}{r}  2 x_1^2 - 3 x_1 + 3 \\ x_2^3 - x_2 \\ x_3^2 - 2 x_3 \\  x_4 - 5 \end{array} \right].
\end{array}
$$
We evaluate the Jacobian of $\true{\vec{f}}(\vec{u})$ in the $N=4$ points ($N$ is chosen such that $N \geq r$) 
$$ 
\vec{u}^{(1)} = \left[ \begin{array}{r} -0.2500 \\ 0 \\ 0.3333 \end{array} \right], \quad 
\vec{u}^{(2)} = \left[ \begin{array}{r} 0 \\ -1 \\ 0 \end{array} \right], \quad 
\vec{u}^{(3)} = \left[ \begin{array}{r} 1 \\ 0.5000 \\ 0.3333 \end{array} \right], \quad 
\vec{u}^{(4)} = \left[ \begin{array}{r} 0.3333 \\ 0 \\ -0.6667  \end{array} \right],
$$
which leads to a $3 \times 3 \times 4$ tensor $\ten{J}$.
The CP decomposition is computed with tensorlab \cite{sorber2014tensorlab2} and returns a rank-four representation with a relative error $\left\| \ten{J} - \hat{\ten{J}} \right\|_F / \left\| \ten{J} \right\|_F$ of $6.40 \times 10^{-14}$ and returns the factors 
$$
\begin{array}{rcl} 
    \mat{V} &=& \left[ \begin{array}{rrrr}   
0.0000 &  -0.3464  &  0.0000  &  1.6749\\
    1.0821 &   0.3464  &  0.0000 &  0.0000\\
    2.1641 &   0.0000 &  -1.6455 &  -1.6749 \end{array} \right], \\ \\
    \mat{W} &=& \left[ \begin{array}{rrrr}  
   0.0000  &  0.0000 &  -1.5072 &  -1.9387 \\
    2.2561  &  0.0000 &   0.0000 &   0.9693 \\
   0.0000  &  0.4803 &   0.0000 &  -0.9693 \end{array}\right], \\ \\
    \mat{H} &=& \left[ \begin{array}{rrrr}   
    0.1365 &  -6.0104 &  -0.5376 &  -3.2850\\
    0.8193 &  -6.0104 &  -0.8064 &  -1.8478\\
    1.2630 &  -6.0104 &  -0.5376 &  -0.2053\\
    1.7751 &  -6.0104 &  -1.3440 &   0.6159 \end{array}\right],
\end{array}
$$
which can be related to the underlying factors $\true{\mat{V}}$, $\true{\mat{W}}$ and $\true{\mat{H}}$. 

Formula~(\ref{eq:Kmin}) tells us that we need $K\geq 5$ points to reconstruct the internal mappings $g_i(x_i)$ so we add 
$$ \vec{u}^{(5)} = \left[ \begin{array}{r} 0.3750 \\ -0.6667 \\ 1.0000 \end{array} \right]$$
to have available $K=5$ points $\vec{u}^{(k)}$ and the corresponding $\vec{y}^{(k)}$.  
We construct the matrix $\mat{R}_K$ of size $15 \times 16$ and verify that its rank equals $15$. 
From the solution of the system~(\ref{eq:yKRKc}) we retrieve the internal functions as 
$$
\begin{array}{rcl}
g_1(x_1) &=& 0.3499 x_1^3 -0.4096 x_1 - 2.2163,
\\ \\ 
g_2(x_2) &=& -6.0104 x_2, 
\\ \\
g_3(x_3) &=& -0.2450 x_3^2 -0.8064 x_3 -6.6347,
\\ \\
g_4(x_4) &=& 0.7355 x_4^2 -1.8478 x_4 + 8.2530.
\end{array}
$$
Ultimately the complete input-output mapping $\vec{f}(\vec{u}) = \mat{W} \vec{g}(\mat{V}^\top \vec{u})$ is reconstructed with a relative error on the coefficients (\ie $\| \vec{c} - \true{\vec{c}} \| / \| \true{\vec{c}} \|$) of $1.3807 \times 10^{-12}$ for $f_1$, $1.7105 \times 10^{-12}$ for $f_2$ and $1.8102 \times 10^{-11}$ for $f_3$.
\end{example}

\subsubsection{Relation $g_i(x_i)$ to $\true{g}_i(\true{x}_i)$} \label{sec:relationgis}
Since the factors $\mat{V}$, $\mat{W}$ and $\mat{H}$ are only identifiable up to scaling and permutation of the columns, the reconstruction of the $g_i(x_i)$ will differ from one representation to the other. 
As it turns out, non-linear relations between the coefficients of $g_i(x_i)$ in the the different (equivalent) representations will show up.

Let us denote by $\mat{V} = \true{\mat{V}} \mat{D}_{\vec{\beta}}$ and $\mat{W} = \true{\mat{W}} \mat{D}_\vec{\alpha}$ the relationship between the representations of the factors $\true{\mat{V}}$ and $\mat{V}$, and $\true{\mat{W}}$ and $\mat{W}$, respectively, which is caused by the column-wise scaling and permutation invariance of the CP decomposition.  
Without loss of generality, we will discard the case of a column permutation in the exposition, implying that $\mat{D}_\vec{\alpha}$ and $\mat{D}_\vec{\beta}$ are diagonal $r \times r$ matrices containing the column-wise scaling factors $\alpha_i$ and $\beta_i$ for $\mat{V}$ and $\mat{W}$, respectively. 
This implies that the $i$-th scaling factors $\alpha_i$ and $\beta_i$ are associated with the $i$-th columns of $\mat{V}$ and $\true{\mat{V}}$, $\mat{W}$ and $\true{\mat{W}}$, and the $i$-th univariate functions $g_i(x_i)$ and $\true{g}_i(\true{x}_i)$. 

We have now that $\true{\vec{f}}(\vec{u}) = \vec{f}(\vec{u})$ and $\mat{V} = \true{\mat{V}} \mat{D}_\vec{\alpha}$ and $\mat{W} = \true{\mat{W}} \mat{D}_\vec{\beta}$, leading to 
$$
\begin{array}{rcl}
        \true{\mat{W}} \true{\vec{g}}(\true{\mat{V}}^\top \vec{u}) &=& \mat{W} \vec{g}(\mat{V}^\top \vec{u}) \\ \\ 
   \true{\mat{W}} \left[ \begin{array}{c} \true{g}_1(\true{x}_1) \\ \vdots \\ \true{g}_r(\true{x}_r) \end{array} \right]                   &=& \true{\mat{W}} \left[ \begin{array}{c} \beta_1 g_1(\alpha_1 \true{x}_1)\\ \vdots \\ \beta_r g_r(\alpha_r \true{x}_r) \end{array} \right]
\end{array}
$$
From the expressions $g_i(x_i) = c_{i,0} + c_{i,1} x_i + \cdots + c_{i,d} x_i^d$ and  $\true{g}_i(\true{x}_i) = \true{c}_{i,0} + \true{c}_{i,1} \true{x}_i + \cdots + \true{c}_{i,d} \true{x}_i^d$ we then find the relation between the coefficients of $g_i(x_i)$ and $\true{g}_i(\true{x}_i)$ as 
\begin{equation}
    \true{c}_{i,\delta} = \beta_i \; \alpha_i^\delta \; c_{i,\delta}.
    \label{eq:coeffrelation}
\end{equation}
\begin{remark}
When $\mat{W}$ is column rank-deficient, the constant terms of the $g_i(x_i)$ cannot be reconstructed uniquely, and the relation (\ref{eq:coeffrelation}) will only hold for $\delta \geq 1$ (see Section~\ref{sec:existuniq}). 
\end{remark}

\begin{example}
For the reconstruction obtained in Example~\ref{ex:runningex3} we can verify that the coefficients of the $g_i(x_i)$ indeed relate to the coefficients of the $\true{g}_i(\true{x}_i)$ through (\ref{eq:coeffrelation}). 
Note that a permutation took place between the columns of the factors, which requires an additional permutation of the scaling factors. 
We have $ \true{g}_1(\true{x}_1) = 2 \true{x}_1^2 - 3 \true{x}_1 + 1$ and $g_2(x_2) = -0.0179 x_2^2 +4.2426 x_2 -2.2369$. 
We verify that 
$$ 
\begin{array}{rclcrcl}
 1 &=& -0.4470 \times -2.2369 & \quad \Longrightarrow \quad &  \true{c}_{2,0} &=& \beta_2 c_{1,0},  \\ \\
-3 &=& -0.4470 \times 15.8175 \times -0.4243 & \quad \Longrightarrow \quad & \true{c}_{2,1} &=& \beta_2 \alpha_2 c_{1,1}, \quad \mbox{and} \\ \\
 2 &=& -0.4470 \times 15.8175^2 \times -0.0179  & \quad \Longrightarrow \quad & \true{c}_{2,2} &=& \beta_2 \alpha_2^2 c_{1,2}. \\ \\ 
\end{array}
$$
A similar analysis can be performed for the relationship between $\true{g}_2(\true{x}_2)$ and $g_1(x_1)$. 
\end{example}

\subsection{Algorithm Summary}
The complete algorithm can be summarized as follows:
\begin{enumerate}
    \item Evaluate the Jacobian matrix $\mat{J}(\vec{u})$ (see (\ref{eq:jacobian})) in $N$ operating points $\vec{u}^{(k)}$ (Section~\ref{sec:matdiag}).
    \item Stack the Jacobian matrices into a three-way tensor $\ten{J}$ of dimensions $n \times m \times N$ (Section~\ref{sec:matdiag}).
    \item Find an appropriate value for $r$ by computing the CP decomposition of $\ten{J}$ (Section~\ref{sec:tensorrank}).  
    \item Retrieve $\mat{V}$, $\mat{W}$ and $\mat{H}$ from the CP decomposition $\ten{J} = \sum_{i=1}^r \vec{w}_i \circ \vec{v}_i \circ \vec{h}_i$ (see (\ref{eq:Jcpd})). 
    \item Reconstruct the internal univariate $g_i(x_i)$ by solving (\ref{eq:yKRKc}).
    \item Check the approximation error of the decoupling procedure, \eg by checking the coefficient-wise errors on the reconstructed $\vec{f}(\vec{u})$.
\end{enumerate}

\section{Open Questions}\label{sec:openprobs}
Several aspects remain to be investigated, such as generalizing the decoupling method to the non-exact case. 
It should be studied how the approximation error can be quantified in a noisy setting: how does noise enter the problem and how are the estimated polynomial coefficients affected by noise. 
This poses the question how this knowledge can be employed in an \emph{(element-wise) weighted CP decomposition} and to what extent the decoupling can be improved.  

Another interesting question is whether \emph{Kruskal's condition} can be loosened by using the knowledge that $\mat{H}$ contains evaluations in $g_i'(x_i)$ (as in (\ref{eq:hki})). 
A tailored CP decomposition, in which such additional information is employed, may guarantee uniqueness up to a greater number of internal functions $g_i(x_i)$.

Although this paper focuses on the polynomial case, the presented method does not require that the non-linear function $\vec{f}(\vec{u})$ is polynomial, neither that the reconstructed univariate $g_i(x_i)$ are polynomial. 
The method can be easily generated to the \emph{non-polynomial case}, as explored in \cite{dreesen2014i2mtc}.

Although it is known that a CP decomposition always exists (given a sufficiently large $r$), a \emph{partial decoupling} that allows cross-linking among a smaller number of variables may be more appropriate. 
One can imagine the case where only some groups of variables can be decoupled from one another, but where an inherent structural (\eg physical) coupling among the variables in a group exists. 
The partial decoupling question suggests the use of the block-term decomposition \cite{delathauwer2008blocktermI, delathauwer2008blocktermII} instead of the CP decomposition.

\section{Conclusions}\label{sec:conclusions}
A method is developed that decomposes a set of multivariate polynomials into linear combinations of univariate polynomials in linear forms of the input variables. 
The paper covers the exact case where a decoupled representation exists and derives a method how to retrieve it.  
The method proceeds by collecting the first-order information of the given functions in a set of Jacobian matrices. 
A simultaneous diagonalization of the Jacobian matrices reveals the linear transformations in the decoupled representation. 
The coefficients of the univariate internal mappings are obtained from the solution of a block-Vandermonde-like linear system of equations that is constructed using the transformation matrices and a set of input-output samples. 

An important advantage of the method is that the curse-of-dimensionality is avoided in the sense that only a third-order tensor is constructed, regardless of the degree of the input polynomials. 
The simultaneous diagonalization of the set of Jacobian matrices is computed by means of the tensor canonical polyadic decomposition, which is known to be unique (up to certain scaling and permutation invariances) under mild conditions.  
It was shown how different (equivalent) decoupled representations are related to one another. 
The different parts of the method were illustrated by means of numerical examples.

\section*{Acknowledgments}
This work was supported in part by the Fund for Scientific Research (FWO-Vlaanderen), by the Flemish Government (Methusalem), the Belgian Government through the Inter university Poles of Attraction (IAP VII) Program, and by the ERC advanced grant SNLSID, under contract 320378. Mariya Ishteva is an FWO Pegasus Marie Curie Fellow.

\bibliographystyle{plain}
\bibliography{refs-arxiv}

\end{document}